\documentclass{amsart}
\usepackage{amsfonts,amscd, amssymb,amsmath,amssymb,enumerate,hyperref,cleveref}
\usepackage{tikz-cd}
\newtheorem{theorem}{Theorem}[section]
\newtheorem{lemma}[theorem]{Lemma}
\newtheorem{corollary}[theorem]{Corollary}
\newtheorem{proposition}[theorem]{Proposition}
\newtheorem{remark}[theorem]{Remark}
\theoremstyle{definition}
\newtheorem{definition}[theorem]{Definition}

\numberwithin{equation}{section}
\makeatother

\DeclareMathOperator{\bR}{{\mathbb R}}
\DeclareMathOperator{\bC}{{\mathbb C}}

\DeclareMathOperator{\Cdb}{{\mathbb C}}
\DeclareMathOperator{\Rdb}{{\mathbb R}}

\DeclareMathOperator{\Hdb}{{\mathbb H}}

\DeclareMathOperator{\Ndb}{{\mathbb N}}

\DeclareMathOperator{\Ml}{{\mathcal M}_\ell}
\DeclareMathOperator{\Al}{{\mathcal A}_\ell}

\DeclareMathOperator{\bT}{{\mathbb T}}

\begin{document}

\title[$M$-ideals in operator  algebras]{$M$-ideals in real operator algebras}
\author{David P. Blecher}
\address{Department of Mathematics, University of Houston, Houston, TX
77204-3008, USA}
\email{dpbleche@central.uh.edu}
\author{Matthew Neal}
\address{Department of Mathematics,
Denison University, Granville, OH 43023}
\email{nealm@denison.edu}

\author{Antonio M. Peralta} \address{Instituto de Matem{\'a}ticas de la Universidad de Granada (IMAG), Departamento de An{\'a}lisis Matem{\'a}tico, Facultad de
	Ciencias, Universidad de Granada, 18071 Granada, Spain.}
\email{aperalta@ugr.es}
\author{Shanshan Su} 
\address{School of Mathematics, East China University of Science and Technology, Shanghai, 200237 China. \\
(Current address) Departamento de An{\'a}lisis Matem{\'a}tico, Facultad de
	Ciencias, Universidad de Granada, 18071 Granada, Spain.}
\email{lat875rina@gmail.com}

\date{Revision of January 22, 2025}

\subjclass[2020]{Primary  46L07,  47L05, 47L25, 47L30, 17C65; Secondary: 
46B04, 46L08, 
47L75, 17C10}

\keywords{Operator space, operator algebra, Jordan algebra, real operator space, $M$-ideal, complexification}

\begin{abstract}  In a recent paper we showed that a subspace of a real JBW$^*$-triple is an $M$-summand if and only if it is a weak$^*$-closed triple ideal. As a consequence, $M$-ideals of real JB$^*$-triples, including real C$^*$-algebras, real JB$^*$-algebras and real TROs, correspond to norm-closed triple ideals.  In the present paper we extend this result by identifying the $M$-ideals in  (possibly non-selfadjoint) real operator algebras and Jordan operator algebras.  The argument for this is necessarily different.   We also give simple characterizations of one-sided $M$-ideals in real operator algebras, and give some applications to that theory. 
  \end{abstract}
  
\maketitle

\section{Introduction}  In 1972 Alfsen and Effros introduced the foundational notions of $M$- and $L$-summands and $M$-ideals of an abstract Banach space \cite{AlfEff72}.  One of their motivations was to generalize to Banach spaces the closed ideals in a C$^*$-algebra and aspects of the theory of such ideals. Alfsen pointed out that 
the $M$-ideals in a complex C$^*$-algebra are just the closed ideals (see  \cite{SW} for a full proof). However, it seems to have taken fifty years to
identify the $M$-ideals in real C$^*$-algebras, and more generally in  real TRO's, real JB$^*$-algebras, and real JB$^*$-triples.  This omission is particularly 
striking since $M$-ideal theory had its origins in the example of $M$-ideals in a real JC-algebra (i.e.\ the selfadjoint part of a JB$^*$-subalgebra of a C$^*$-algebra; see Proposition 6.18 in conjunction with line 6--7 on p.\ 129 of \cite{AlfEff72}). 

The formidable difficulties with identifying the $M$-ideals in such spaces were overcome in the recent paper \cite{BNPS}.  It is shown there using deep facial structure that $M$-ideals of these  objects correspond to norm-closed triple ideals. That is, purely geometric objects ($M$-ideals) are algebraically characterized. In the present paper we extend this result to (possibly non-selfadjoint) real operator algebras (that is, closed subalgebras of the algebra of bounded operators on a real Hilbert space), and to real Jordan operator algebras. (The reader who does not care about nonassociative algebras should simply ignore the word `Jordan'.)  That is, we identify the $M$-ideals in these algebras. Here the argument is necessarily different.  Indeed we were able to find a much shorter proof in this case, which is quite surprising given the general technical obstacles in the real $M$-ideal case.  (For example, there may not exist enough positive or selfadjoint (hermitian) elements, or projections.  See \cite{BNPS} for a discussion of some of the other difficulties.) We also discuss  one-sided $M$-ideals in real operator algebras, and characterize these objects metrically and algebraically. 
 
Thus our paper takes its place in a long line of results on  when  the metric or norm structure on a space, a geometric condition, forces its  algebraic structure. The definition of an $M$-ideal of a real operator algebra $A$ is purely geometric, depending only on the vector space structure and the norm, and yet we show that it characterizes the ideal structure of $A$. Abstract $M$-ideals 
constitute a formidable technology (see \cite{HWW}),
and have become, undoubtedly, one of the most important tools in the isometric theory of Banach spaces, and a highly 
influential notion in other branches of mathematics. Such results are not mere abstract games without concrete applications. We mention some common directions of application like understanding more deeply basic objects in deep modern theories, including noncommutative geometry or quantum information, by capturing important operations and substructures in geometric terms.  
They can be used to throw new light on the structure and properties of mathematical 
objects when we do not yet have or know their full structure. 
As just one  example, $M$-projections are key to the ``intrinsic structure'' in complex operator spaces 
and operator systems (see e.g.\ 
\cite[Section 3]{BReal}). Generalizing powerful and foundational existing theory to larger classes of operator algebras, may pave the way to tackle problems arising in mathematical physics, and current cutting-edge research supports the idea of exploring beyond the traditional framework. For example, in a recent work Connes and van Suijlekom extend the traditional noncommutative geometry to operator systems in order to, among other things, deal with spectral truncations of certain geometric spaces, or to incorporating very general tolerance relations (see e.g.\ \cite{CvS} and its sequels).

This article is structured as follows.  We will begin with some  background and notation.  Since our paper is concerned in many places with real Jordan operator algebras, we will need some facts from that theory which we develop  in Section \ref{joas}. 
Section \ref{Mids} contains our main result on $M$-ideals in real operator algebras, Jordan or otherwise, 
together with some sample applications.   
If the algebra is a real C$^*$-algebra we recover the main result of \cite{BNPS} in this special case, but with a completely different proof.  We remark that this proof however does not generalize to the setting e.g.\ of real TRO's or JB$^*$-triples. 
In Section  \ref{osmi} we give simple characterizations of one-sided $M$-summands and $M$-ideals in real operator algebras, and we give some applications to that theory.

The reader will need to be familiar with some of the  basics of complex operator spaces and von Neumann algebras  
as may be found in early chapters of e.g.\  \cite{BLM,ER} and  \cite{P}. 
It may also be helpful to browse the (small) existing real operator space theory, e.g.\ \cite{ROnr,RComp,Sharma,BT,BReal}.  
A note on references: usually we will reference the real case of a result in the literature, but often these sources simply check that the complex 
argument (where usually the serious work was done) transfers. 
Some basic 
real C$^*$-algebra theory may be found in \cite{Li} (see also  \cite{Good}).

We now turn to notation.  For us a {\em projection}  in a real or complex C$^*$-algebra  is always an orthogonal projection (so $p = p^2 = p^*$). A projection in a normed algebra or in a Jordan algebra $A$ is a contractive idempotent element (this notation will not conflict with the one above in the settings of this paper).  We say
$A$  is {\em unital} if it has a unit $\mathbf{1}$ of norm $1$, 
and a map $T$ 
is unital if $T(\mathbf{1}) = \mathbf{1}$. 
 We will use silently the fact from basic functional analysis that for a subspace  $X$ of a Banach space 
$Y$,
$X^{\perp \perp}$ is the weak$^*$ closure of $X$ in $Y^{**}$, and is isometric to $X^{**}$. 
An element $a$ of a real C$^*$-algebra is  {\em real positive} if $a + a^* \geq 0$. 
 
 Let $V$ be a Banach space. A projection $P$ on $V$ is called an \emph{$L$-projection} (respectively, an \emph{$M$-projection}) if  $$\|  x \| = \| P(x) \| + \| (I -P)(x) \|, \quad \hbox{ for all } x \in V$$ (respectively, $\|  v \| = \max \{ \| Pv \| , \| (I-P)(v) \|  \}$  for all $v \in V$).  The image of an $L$-projection (respectively, an $M$-projection) on $V$ is called an \emph{$L$-summand} (respectively, an \emph{$M$-summand}).  A closed subspace $M$ of a Banach space $V$ is said to be an \emph{$M$-$ideal$} if $M^{\perp \perp}$ is an $M$-summand of $V^{**}$.  See \cite{HWW} for the theory of $M$-ideals. 
 
 A  \emph{real TRO} is a closed linear subspace $Z \subseteq B(K, H),$ for real Hilbert spaces $K$ and $H$, satisfying $Z Z^* Z \subseteq Z$. 
 
 Complex operator spaces are an important 
umbrella category containing C$^*$-algebras, operator systems, operator algebras, von Neumann algebras, TRO's, 
and many other objects of interest  (see e.g.\ \cite{ER,BLM}). 
Ruan initiated the study of real operator spaces in \cite{ROnr,RComp}, and this study was continued in e.g.\  \cite{Sharma, BT,BReal,BCK}.   
A real operator space may either be viewed as a real subspace of $B(H)$ for a real Hilbert space $H$, or abstractly as 
a vector space with a norm $\| \cdot \|_n$ on $M_n(X)$ for each $n \in \Ndb$ satisfying  the conditions of
Ruan's characterization in  \cite{ROnr}.    In \cite{BReal} we verified that a large portion of the theory of complex 
operator spaces and operator algebras (as represented by the text {\rm \cite{BLM}} for specificity)
transfers to the real case.

If $T : X \to Y$ is a linear map between operator spaces we write $T_n$ for the canonical `entry-wise' amplification taking $M_n(X)$ to $M_n(Y)$.   
We say that $T$ is {\em  completely bounded}  if $\| T \|_{\rm cb} = \sup_n \, \| T_n \| < \infty$.  Then $T$ is 
{\em completely  contractive} (resp.\ {\em completely isometric}) if  $\| T \|_{\rm cb}  \leq 1$ (resp.\ $T_n$ is an isometry for all $n$).  The space CB$(X,Y)$ of such completely bounded 
maps is an operator space, and of course is a unital algebra
if $X = Y$.   An operator space $X$ has  a norm  on  the rectangular matrices over $X$.
This norm is easy to describe for $C_n(X) = M_{n,1}(X)$ and $R_n(X) = M_{1,n}(X)$.  E.g.\  $\| [x \; \; y ] \|^2_{R_2(X)}  = \| x x^* + y y^* \|$ by the C$^*$-identity with respect to any C$^*$-algebra containing $X$ completely isometrically.

A projection $P$ on an operator space $X$  is a {\em complete} $M$-{\em projection} if $P_n$ is an $M$-projection
on $M_n(X)$ for all $n \in \Ndb$.  Its range is a {\em complete} $M$-{\em summand}, and $J$ is a {\em complete} $M$-{\em ideal} in $X$ if
$J^{\perp \perp}$ is a  complete $M$-summand of $X^{**}$.   Several alternative characterizations of these objects may be found in e.g.\
\cite[Section 4.8]{BLM} (see also \cite{BT,BReal}). It is known that  $M$-ideals in 
complex C$^*$-algebras and in unital operator algebras are automatically complete $M$-ideals, and our current paper(s) establishes the real case. However, this is not always true for operator spaces (cf.\ \cite[Proposition 4.1]{ER1994}).

 An {\em operator space complexification} of a real operator space $X$ 
is a pair $(X_c, \kappa)$ consisting of a complex operator space $X_c$ and a real linear complete isometry $\kappa : X \hookrightarrow X_c$ 
such that $X_c = \kappa(X) \oplus i \, \kappa(X)$ as a vector space.   For simplicity we usually identify $X$ and $\kappa(X)$ and write $X_c = X + i \, X$.
We say that the complexification is {\em reasonable} if the map 
$\kappa(x) + i \kappa(y) \mapsto  \kappa(x) - i \kappa(y),$ for $x, y \in X$, is 
a complete isometry on $X_c$.  Ruan proved that a  real operator space has a unique reasonable complexification
$X_c = X + i X$ up to complete isometry \cite{RComp}.  
   If $T \in CB(X,Y)$ then we write $T_c : X_c \to Y_c$ for the 
canonical complexification of $T$.
We have $\| T_c \|_{\rm cb} = \| T \|_{\rm cb}$. 
    
A  {\em skew} element of a $*$-algebra is an element  with $x^* = -x$.   Any element of the $*$-algebra is the sum of a skew and a selfadjoint element. 
    A skew real positive element need not be zero:  in $M_2$ the matrix $a = \left[ \begin{array}{ccl}  0 & 1 \\ -1 & 0 \end{array} \right]$ is a counterexample.   Moreover, this fails even if it has a square root which is real positive,
as does this $a$. 
 
By a real (resp.\ complex) {\em  Jordan operator algebra} we 
 mean a  norm-closed  real (resp.\ complex) {\em  Jordan subalgebra} $A$ of a real (resp.\ complex) C$^*$-algebra,  
namely a norm-closed   real (resp.\ complex) subspace closed under the 
`Jordan product' $a \circ b = \frac{1}{2}(ab+ba)$. Or equivalently, 
 with $a^2 \in A$ for all $a \in A$ (this follows since $a \circ b = \frac{1}{2} ((a+b)^2 -a^2 -b^2)$). 
 There do exist abstract characterizations of such algebras \cite{BWjmn,BNjmn,BT,Tepsan}, but we will not take the time to state them here. 
The selfadjoint case, that is,  closed selfadjoint real (resp.\ complex) subspaces of a real (resp.\ complex) C$^*$-algebra which are closed under squares, are called real (resp.\ complex) {\em JC$^*$-algebras}. Complex JC$^*$-algebras have a large literature, see e.g.\ \cite{CabRodvol1,Wright1977} for references.  Also JC-algebras (i.e., the selfadjoint parts of JC$^*$-algebras, cf. \cite{Wright1977}) are real Jordan operator algebras. The theory of (possibly non-selfadjoint) Jordan operator algebras over the complex field was initiated in \cite{BWjmn,BNjmn,BWj2}. If  $A$ is a Jordan operator subalgebra  of $B(H)$, then the {\em diagonal} $\Delta(A) = A \cap A^*$  is a JC$^*$-algebra.   The bidual $A^{**}$  of an (Jordan) operator  
algebra is again an (Jordan) operator algebra with the Arens product, containing $A$ as a (Jordan) subalgebra. 

All real JC$^*$-algebras are contained in the strictly wider class of real JB$^*$-algebras, as considered in \cite{Alvermann,Pe2003ax,FerMarPe2004, Rod2010,ApaPe2014}. Real and complex C$^*$-algebras, complex JB$^*$-algebras, and JB-algebras are all examples of real JB$^*$-algebras. However, it should be noted that there are examples of exceptional JB- and JB$^*$-algebras which cannot be embedded as Jordan self-adjoint subalgebras of C$^*$-algebras (see, for example, \cite{HOS,CabRodvol1} for a more detailed account). We have indirectly  employed that, by a result of J.D.M. Wright \cite{Wright1977}, JB-algebras are nothing but the real Jordan--Banach algebras determined by the self-adjoint parts of JB$^*$-algebras. Because of the obvious obstacles from the point of view of operator spaces, the wider category of JB$^*$-algebras remains beyond the goals of this paper.

Recall that a closed subspace $I$ of a real or complex JB$^*$-algebra $A$ is a (Jordan) ideal of $A$ if $I\circ A \subseteq I$. It is known that Jordan ideals of $A$ are all self-adjoint (see, for example, the discussion in \cite[page 8]{GarPe2021} or \cite[Proposition 3.4.13]{CabRodvol1}). The quotient of a 
(complex) JC$^*$-algebra by a closed ideal is again a JC$^*$-algebra (see \cite[Proposition 4.7.4]{HOS} or explicitly in \cite[Lemma 3.1]{Wright1977}), and thus, by passing through the complexification, the same conclusion remains true for real JC$^*$-algebras. 

The next question which arises is whether to treat real Jordan operator algebras as Banach spaces or as 
 real  operator  spaces.   In this connection we remind the reader  that the 
 complexification of a  real operator algebra is not uniquely defined up to isometry  (see \cite[Proposition 2.10]{BT}), but is uniquely defined up to complete isometry.
 For this reason we will usually treat all algebras as operator spaces for safety, so that when, for example, we view a (Jordan) operator algebra as a subalgebra
 of a C$^*$-algebra  we mean that it is completely isometrically embedded.   This is probably not necessary, 
perhaps nearly all of our results  are easily adaptable to the Banach setting.  Indeed usually the issue does not  arise since any concrete Jordan operator algebra in $B(H)$ 
already has an operator space structure.  

\section{Preliminaries on real Jordan operator algebras} \label{joas} 

The reader only interested in associative algebras could skip most of this section, particularly if they are familiar with real associative 
operator algebras. 

The basic theory  of (possibly non-selfadjoint) complex Jordan operator   algebras may be found in \cite{BWjmn,BNjmn,BWj2}.    The real case of this theory was initiated in 
 \cite{BT,Tepsan}.   In addition to reviewing some aspects of the latter  theory,
 we will focus in this section on results that  we need which  were omitted or not explicit in the latter sources. 
As just mentioned, we will assume for convenience
that (Jordan) operator algebras are completely isometrically represented on a Hilbert space. 
Thus for us a C$^*$-{\em cover} of an operator algebra 
 (resp.\  Jordan operator algebra) $A$ is a
 pair $(B,j)$ consisting of a real C$^*$-algebra $B$ and a completely isometric homomorphism (resp.\  Jordan homomorphism) $j : A \to B$ such that $j(A)$ generates
 $B$ as a real C$^*$-algebra.   There exists a `smallest'   C$^*$-cover, the C$^*$-envelope $C^*_e(A)$, sometimes called the {\em noncommutative Shilov boundary} (see e.g.\ \cite[Corollary 4.3]{BCK} 
 and \cite{Tepsan} for the real case of this).
 
 We say that  an operator algebra $A$ is {\em approximately unital} if it has a contractive  approximate identity (cai).
 The situation is more complicated for a Jordan operator algebra, here there are partial cai's, J-cai's, etc.\
 for which we refer the reader to  \cite[Section 4]{BT} in the real case.   Fortunately such an algebra has a partial cai  if and only if it has a  J-cai, and 
 if and only if the bidual $A^{**}$ has an identity of norm 1.  In this case we say that $A$ is  {\em approximately unital}. 
 For example, all JC$^*$-algebras are approximately unital (see e.g.\ \cite{CabRodvol1}). 
 The unitization $A^1$ of an approximately unital (Jordan) operator algebra  is well defined up to isometric or completely isometric isomorphism 
 \cite[Corollary 3.4 and Lemma 4.10]{BT}, and we  have $(A^1)_c = (A_c)^1$. 
  
 It is well known that a (Jordan) $*$-homomorphism  on a real JC$^*$-algebra has closed range. Actually, $*$-homomorphisms between C$^*$-algebras, Jordan $*$-homo-morphisms between JB$^*$-algebras, and triple homomorphisms between JB$^*$-triples are automatically continuous and contractive \cite[Lemma 1]{BarDanHorn88}. Having in mind that these maps extend in a natural way to the complexifications, the same conclusions hold for the corresponding real versions of these structures. An extension of the Cleveland--Kaplansky theorem to JB$^*$-triples asserts that a  triple homomorphism $T$ from a real JB$^*$-triple to a normed Jordan triple has closed range whenever it is continuous, and $T$ is bounded below if and only if $T$ is a triple monomorphism \cite{FerGarPe2012}.
	
	The celebrated Kaup-Kadison-Banach-Stone theorem assures us that a linear bijection between JB$^*$-triples is an isometry if and only if it is a triple isomorphism \cite[Proposition 5.5]{Kaup1983}.  This conclusion  is not always true in the real setting \cite{FerMarPe2004}. However, the difficulties do not appear in the setting of real JB$^*$-algebras, where it is known that a linear bijection between real JB$^*$-algebras, in particular between real JC$^*$-algebras, is an isometry if and only if it is a
	 triple isomorphism (see e.g.\ \cite[Corollary 3.4]{FerMarPe2004} or \cite[Theorem 4.8]{IKR}).

The diagonal $\Delta(A) = A \cap A^*$ of a real (Jordan) operator algebra $A$ (with possibly no kind of identity) is a well-defined real  JC$^*$-algebra independent of representation, just as in the complex case.  Indeed by \cite[Theorem 2.6]{BT}, a contractive (Jordan) homomorphism (resp.\ isomorphism) $\pi : A \to B$ between real (Jordan) operator algebras restricts to a (Jordan) $*$-homomorphism (resp.\ isometric $*$-isomorphism)  from the diagonal $\Delta(A)$  into a (Jordan) $*$-subalgebra of $\Delta(B)$. Thus, the diagonal $\pi(\Delta(A))$ is a real JC$^*$-subalgebra. It follows that the appropriate versions of
the results in  2.1.2 in \cite{BLM} hold in the real (Jordan) case. In particular, the definition of $\Delta(A)$ does not dependent of the particular representation of $A$ (up to $*$-isomorphism). For a real (Jordan) operator algebra $A$ 
with a canonical operator space structure (so completely isometrically embeddable as a (Jordan)  subalgebra
 of a C$^*$-algebra $B$),  we have that 
$\Delta(A_c) = A_c \cap A_c^* \cong \Delta(A)_c$ $*$-isomorphically. 
Indeed all these spaces may be identified with the same $*$-algebra of $B_c$. 
 Because the selfadjoint elements do not necessarily span a real  C$^*$-algebra 
 (or may be all of the C$^*$-algebra), one needs to be careful in places.  
So  $\Delta(A)$ need not be the span of the selfadjoint elements 
 (nor of the projections if $A$ is also weak$^*$ closed),
 unlike in the complex case.

We next discuss real Jordan multiplier algebras.   The complex case was considered in \cite{BWjmn} (mostly in Sections 2.8 there), and around Lemma 1.2 in \cite{BNjp}.   Section 4.7 of \cite{Tepsan} discusses briefly the left and right multiplier algebras of an approximately unital real Jordan operator algebra $A$ but does not claim the 
  real case of all of the assertions of \cite[Theorem 2.24]{BWjmn}.        
  We claim much more here.  One important principle is that  
  two elements $x$ and $y$ in an approximately unital real Jordan subalgebra $A$ of a C$^*$-algebra  $B$ may be multiplied using the associative product in  the C$^*$-algebra--which is denoted by mere juxtaposition,  
  and if this product $xy$ is in $A$ then it is independent of the particular C$^*$-algebra $B$  containing $A$ (such as the C$^*$-envelope $C^*_e(A)$).  We call this particular product the 
``C$^*$-{\em product}'' below, and we shall denote it by mere juxtaposition.  Note that we are using the operator space structure of $A$ in this definition.  
We may then define the left multiplier algebra LM$(A) = \{ \eta \in A^{**} : \eta A \subset A \}$ (where the product is the C$^*$-product on $A^{**} \subset B^{**}$).   
Note that LM$(A)  \subset {\rm LM}(A_c)$ as a real subalgebra, indeed ${\rm LM}(A_c) = {\rm LM}(A)_c$ completely isometrically (this may be deduced, for example, immediately  from the next result and \cite[Lemma 4.2]{BReal}). 

A mapping $u$ on a real operator space $X$ is called \emph{a left multiplier of
$X$} if there exists a linear complete isometry $\sigma : X \to B(H)$ for some real Hilbert space $H$, and an operator $S \in B(H)$ such that
$\sigma (u(x)) = S \, \sigma (x),$ for all $x \in X$. The set of all left multipliers of $X$ will be denoted by 
$\Ml (X)$ (see \cite[\S 5]{Sharma} and \cite[Section 4]{BReal}).  A mapping $u\in \Ml(X)$ is said to be \emph{left adjointable} if we can find $H, S, \sigma$ as above satisfying
 in addition $S^* \sigma (X) \subseteq \sigma(X)$. The collection of all left adjointable maps on $X$ is denoted by $\Al(X)\subseteq \Ml(X)$. It is known that $\Ml(X)$ is a unital subalgebra of a C$^*$-algebra and $\Al(X)$ is precisely the diagonal 
$C^*$-algebra, $\Delta(\Ml(X)) = \Ml(X) \cap \Ml(X)^*,$ of $\Ml(X)$. We therefore have a well-defined involution $u\mapsto u^*$ on $\Al(X)$, in which $u^*$ is determined by $\sigma,\sigma^{-1}$ and $S^*$ in a natural way (cf. \cite[Proposition 4.3]{BReal}).   

By the arguments in  \cite[Theorem 2.24]{BWjmn}, but using real variants from \cite{BT,Tepsan} of some  facts from the complex theory, we have:
 
 \begin{theorem} \label{tlma}
Let $A$ be an approximately unital real  Jordan operator algebra and let
$B = C^*_e(A)$, with $A$ considered as a Jordan  subalgebra. 
Then ${\rm LM}(A) \cong \{\eta\in B^{**} :  \eta A\subset A \}$.   This is  
completely isometrically isomorphic to the (associative) operator algebra ${\mathcal M}_\ell(A)$ of
operator space left multipliers of $A$
in the sense of e.g.\ {\rm \cite[Section 4]{BReal}}, and is completely isometrically isomorphic to 
a unital subalgebra of $CB(A)$.   Also, $\Vert T \Vert_{\rm cb} = \Vert T \Vert$ for every 
$T \in {\rm LM}(A)$ thought of as an operator on $A$.  
 \end{theorem}

{\bf Remark.} Theorem 2.24 in  \cite{BWjmn} also asserts that for any nondegenerate completely isometric Jordan representation $\pi$ of $A$ on a Hilbert space $H$, the algebra $\{T\in B(H): T\pi(A)\subset \pi(A)\}$ is completely isometrically isomorphic to 
a unital subalgebra of LM$(A)$, and this isomorphism maps onto  LM$(A)$ if $\pi$ is a faithful nondegenerate  $*$-representation of $B$ for example.
This appears to be valid in the real case too with the same proof. 
\medskip

Similar results hold for the right multipliers RM$(A)$.  We define the multiplier algebra M$(A) = {\rm LM}(A) \cap  {\rm RM}(A)$
as in the complex case.  The {\em Jordan multiplier algebra} JM$(A)$ of $A$ is defined to be
the set of elements $\eta \in A^{**}$ with $a \circ \eta =\frac12(\eta a + a \eta) \in A$ for all $a \in A$. We observe that we do not  need the ``external'' C$^*$-
product to define the Jordan multiplier algebra, and that this notion is precisely the usual Jordan multiplier algebra in the available literature \cite{Ed80}.  

As in 
 Lemma 1.2 in \cite{BNjp}  one sees that JM$(A)$ is a real unital Jordan operator algebra containing $A$ as 
 an approximately unital Jordan ideal.  It also contains M$(A)$.    We shall not use this later, but by the definition one can
 see that  ${\rm JM}(A_c) = {\rm JM}(A)_c$, and a similar relation holds 
 for M$(A)$. 
 
 We will need the following result, which in the real case is essentially in  \cite{Tepsan}.

 \begin{lemma} \label{isbab} Let $A$ be an approximately unital real Jordan operator algebra.
If $p$ is a projection in ${\rm LM}(A)$ then $p \in {\rm M}(A)$.   More generally, the diagonal $\Delta({\rm LM}(A)) \subset {\rm M}(A)$.   \end{lemma}
 
 \begin{proof}  This is proved by the same argument in  \cite[Theorem 2.23]{BWjmn}, once one knows that Lemma 2.1.6 in 
 \cite{BLM} holds in the real case.  The latter was confirmed in \cite[page 23]{BReal} (see also \cite{Tepsan}).
   \end{proof}  

Let $X$ be a real (or complex) operator
space. A contractive projection $P: X\to X$ is called a \emph{complete left $M$-projection on $X$} (also called a \emph{left $M$-projection on $X$}) if the operator $\nu^c_{P}
: X \to C_2(X)= M_{2,1} (X)$,  $x \mapsto \left(\begin{matrix}
P(x) \\
x -P(x)
\end{matrix} \right)$ is a complete isometry (see also Section~\ref{osmi} for more details).

The complete left $M$-projections on $X$ are also exactly the projections in the algebra $\Ml(X)$ above (see \cite[Proposition 4.1]{BEZ}, \cite[Proposition 5.3]{Sharma}). More concretely, the complete left $M$-projections on $X$ are precisely the operators of the form $P(x) = e\sigma(x)$ for a completely isometric embedding $\sigma: X \hookrightarrow B(H)$ and an orthogonal projection $e$ in $B(H)$. 
Here $x \in X$. Therefore, Theorem \ref{tlma} and Lemma \ref{isbab} allow us to identify immediately the complete left $M$-projections on an approximately unital real Jordan operator algebra $A$.  The corresponding summands are the subspaces $eA$ for a projection $e$  in 
M$(A)$.  We 
give more details in Section 
\ref{osmi}.
 
States  on a unital real Jordan operator algebra are, by definition, the unital contractive functionals. 
  States on an approximately unital real Jordan operator algebra $A$ are discussed in \cite{BT} (they are defined at the end of Section 4 there). More precisely, a norm-one functional $\varphi\in A^*$ is a state if  $\varphi(e_j) \to 1$ for every (equivalently, some) partial cai $(e_j)$ in $A$.  It is equivalent to use  Jordan cai in place of partial cai in this definition. It is also known that it is equivalent to say that $\varphi (\mathbf{1}) = 1$ where $\mathbf{1}$ is the identity of $A^{**}$ \cite[Lemma 4.13]{BT}.  
  
  In the proof of our main theorem we will need to use such states, and their relation to states on the complexification.   The following clarifies this relationship.

\begin{lemma} \label{rsrp} The real states $\varphi$ on an approximately unital real Jordan operator algebra $A$ are precisely the 
restrictions to $A$ of real parts   
of complex states  on $A_c$ (such as $\varphi_c$), or on a real
(resp.\ complex) C$^*$-cover of $A$ (resp.\ $A_c$).   \end{lemma}
 
   \begin{proof}  
If $\varphi$ is a real state on $A$ then $\varphi_c :  A_c \to \mathbb{C},$ $\varphi_c (a+ i b ) := \varphi(a) + i \varphi(b)$ ($a,b\in A$) has norm 1 and is a complex state.  This is obvious if $A$ is unital. 
In the approximately unital case this follows for example  using \cite[Lemma 4.13]{BT}, since there is a partial cai for $A$ which is also one for $A_c$, and is a cai for 
any C$^*$-cover of $A_c$.     
The converse is similar; and the rest is left as an exercise (using e.g.\ the statement after \cite[Lemma 4.2 (iv)]{BT}).   
  \end{proof}  

{\bf Remark.} The real parts of two different complex states on $A_c$ may coincide on $A$.   For example consider the normalized trace $\frac{1}{2} \, Tr$ on  $M_2(\Cdb)$ and the state
$Tr( \cdot \, p)$, where $p$ is the rank 1 projection associated with the vector $\xi = \frac{1}{\sqrt{2}} \, [1 \; i ]$.   On $M_2(\Rdb)$ (or on the real upper triangular 
$2 \times 2$ matrices) the real parts of these two states coincide.
{That is, the mapping $\varphi\mapsto \Re\hbox{e} \, \varphi|_{A}$ is not injective when considered from $A_c^*$ to $A^*$.}  

\section{M-ideals in operator algebras}  \label{Mids} 

The following is known, and is true  in much greater generality.  (See e.g.\
\cite[Lemma 1.6]{FR}, or more explicitly in \cite[Identity $(6)$ in page 360]{FerMarPe2012} in the complex case, from which the real case can be deduced. There is also  direct proof in the setting of real C$^*$-algebras in \cite[Lemma 3.1 and proof of Theorem 3.2]{FerPeSurvey2018}.  The latter is close to the proof that we present here for the sake of completeness.)

  \begin{lemma} \label{orttr}   If $z$ is a tripotent in a real TRO $Z$ (that is, $z = z z^* z$),
   and $\| z \pm b \| \leq 1$ for some $b \in {\rm Ball}(Z)$ then $b = (1-z z^*) b (1 - z^* z)$.   \end{lemma}
 
   \begin{proof}   
Pre- and postmultiplying by projections we have   $\| z \pm z z^* b z^* z  \| \leq 1$.   Hence   $z z^* b z^* z = 0$. 
This is clear if $z$ is a projection, but a similar argument works for tripotents considering $z z^*  Z  z^* z$ as a C$^*$-algebra with product $x z^* y$ and identity $z$.  
Similarly, left multiplying by the  projection $z z^*$  we have  $$\| z \pm z z^* b (1 - z^* z) \| = \| z \pm z z^* b  \| \leq 1.$$ 
So  $\| (z + z z^* b (1 - z^* z) )(z^* + (1-z^* z) b^* z z^* ) \| \leq 1$. 
That is, $$\| z z^* + z z^* b (1 - z^* z)  b^* z z^*) \| \leq 1,$$ which forces 
$z z^* b (1 - z^* z)  = 0$.   In a similar way, $(1-z z^*) b z^* z = 0$. 
 Thus $b = (1-z z^*) b (1 - z^* z)$. 
  \end{proof}  
  
 To say that a projection $p$ in a Jordan operator algebra $A$ commutes in the  C$^*$-product with $x \in A$ may be written in Jordan algebra terms as $$p \circ x  = pxp = 2 (p\circ x) \circ p - p \circ x.$$
 In this case we simply say that  $p$ {\em commutes with} $x$.   If $p$ commutes with all elements of $A$ we shall say that $p$ is {\em central}. 
  It is easy to see that a projection is central in  JM$(A)$ if and only if it commutes with all elements of $A$, assuming that $A$ 
  is approximately unital (cf. \cite[Lemma 4.2, final statement]{BT}).

\begin{theorem} \label{mid}  
Let $A$ be an approximately unital real  operator algebra (resp.\ Jordan operator algebra). 
\begin{enumerate}[{\rm(1)}] 
\item The $M$-ideals in $A$ are the complete 
$M$-ideals. These are exactly the 
closed ideals (resp.\ closed Jordan  ideals)  $J$ in $A$ which are approximately unital.   Moreover in the Jordan case  they are 
also ideals in $A$ in the C$^*$-product: $AJ + JA$ and $AJA$ are contained in $J$.
\item The $M$-summands in $A$ are the complete $M$-summands. These are exactly the ideals $Ae$ for a 
projection $e$  in ${\rm M}(A)$ (resp.\  ${\rm JM}(A)$ or equivalently  in ${\rm M}(A)$) such that $e$ commutes with all elements in $A$.     
 \end{enumerate}
\end{theorem}

\begin{proof} (1) \  First suppose that $A$ is a unital real Jordan operator algebra, inside a unital real C$^*$-algebra $B$ which it generates. 
Let $P$ be an $M$-projection on $A$, and let $z = P(\mathbf{1})$.   
For any state $\varphi$ of $A$ we have $\varphi(z) = \| P^*(\varphi) \| \geq 0$ by the argument  for (4.19) in  the proof of Theorem 4.8.5 (2) in  \cite{BLM}. 
For completeness we repeat the argument. Since $-1 \leq \varphi (P(\mathbf{1})),$ $ \varphi ((I-P)(\mathbf{1})) \leq 1,$ since $P^*$ is an $L$-projection, and  since $$1= \varphi (P(\mathbf{1})) +  \varphi ((I-P)(\mathbf{1})) \leq \|P^*(\varphi)\| + \|(I-P)^* (\varphi)\| =\|\varphi\| = 1, $$ we have $0\leq \varphi(z) =\varphi (P(\mathbf{1})) = \|P^*(\varphi)\|$. 

We have $\| \mathbf{1}  - z \| = \| (I-P)(\mathbf{1}) \| \leq 1$.  Taking adjoints we have  $\| \mathbf{1} - z^* \|  \leq 1$, so that 
$\| \mathbf{1} - \frac{z+z^*}{2} \| \leq 1$.  It follows that  $a= \frac{z+z^*}{2}$ is positive.   
Write $z = a + i \frac{z-z^*}{2i}= a + ib \in B_c$ --in the complexification of $B$--, with $a \geq 0$ and $b = b^* =  \frac{z-z^*}{2i}$ (note that $b$ is not necessarily in $B$ although $b^2 \in B$).

Suppose that $\varphi$ is  a state  of $B$ with $\varphi(z) = 1$.    
Then $\rho = \varphi_c$ is a state of $B_c$.  We have  $1 = \varphi(z) = \rho(a) + i \rho(b)$, and so
$\rho(a) = 1$ and $\rho(b) = 0$.  We have  $z^2 = a^2 - b^2 + i (ab+ba)$.  Now $1 = \rho(a)^2 \leq \rho(a^2) \leq 1$ by the Cauchy-Schwartz inequality, and so $\rho(a^2) = 1$.
Similarly, $1 = \varphi(z)^2 \leq \varphi(z^* z) \leq 1$, and $$1 = \varphi(z^* z) =   \rho(a^2 + b^2 + i(ab-ba)) = \rho(a^2 + b^2) 
+ i \rho(ab-ba).$$ 
 Similarly $$1 =\varphi(zz^*) =   \rho(a^2 + b^2 - i(ab-ba)) = \rho(a^2 + b^2) 
- i \rho(ab-ba),$$
so that
$\rho(a^2 + b^2) = 1,$ $\rho(b^2) = 0,$  and $\rho(ab-ba)=0$.   Hence  $$\varphi(z^2) = \rho(a^2 - b^2 + i (ab+ba)) = 1 + i \rho(ab+ba) .$$
Since $|\varphi(z^2) | \leq 1$ and $\rho(ab+ba) \in \bR$ this forces  $\varphi(z^2) = 1$.  Hence  $\varphi(z(\mathbf{1}-z)) = 0$.  

We have shown that $\varphi(z(\mathbf{1}-z)) = 0$ for every state $\varphi$ of $B$ with 
$\varphi(z)=1.$ We can bootstrap from this to show that the same formula  also holds for any state $\varphi$ of $A$ or $B$, or equivalently, for the real part of every state on $B_c$, by 
following the argument to prove  (4.20) in the proof of  Theorem 4.8.5 (2) in  \cite{BLM}. We include the  argument here for completeness. Consider  any state $\varphi$ of $A$. 
 Note that $\varphi(z) =0$ if and only if $\varphi (\mathbf{1}-z)=1$.  Having in mind that $\mathbf{1}-z = (I-P) (\mathbf{1})$ and $(I-P)$ is also an $M$-projection, the previous case gives $\varphi ((\mathbf{1}-z) z) = 0$.  We may therefore assume that $0\neq \varphi (z).$  As we have seen above, $\varphi(z) =\varphi (P(\mathbf{1})) = \|P^*(\varphi)\|,$ and thus $\psi :=\|P^*(\varphi)\|^{-1} P^*(\varphi)$ is a state of $A$. It is easy to check that $$\psi (z) = \|P^*(\varphi)\|^{-1} \varphi(P(z))= \|P^*(\varphi)\|^{-1} \varphi(P^2(\mathbf{1})) = \|P^*(\varphi)\|^{-1} \varphi(z) =1.$$ We conclude from what we know that $\psi (z(\mathbf{1}-z)) = 0,$ equivalently, $\varphi (P (z(\mathbf{1}-z))) = 0.$ By replacing $P$ with $(I-P)$ and $z$ with $\mathbf{1}-z = (I-P)(\mathbf{1})$, we arrive to $\varphi (I-P) (z(\mathbf{1}-z)) = 0.$  Combining this with what we have just proved gives $\varphi (z(\mathbf{1}-z)) = 0$ 
 as claimed. 

Since $\varphi(z(\mathbf{1}-z)) = 0$ for every state $\varphi$ of $B$, it follows from Lemma \ref{rsrp} that $\Re\hbox{e} \, \rho(z(1-z)) = 0$ for any state $\rho$ of $B_c$. That is, $d = z (1-z)$ is ``skew symmetric'': $d^* = -d$, or $d+ d^* = 0$.  Now $d + d^*=  2(a(1-a)+b^2)$, and so
$b = 0 =a(1-a)$, and $a = a^2$ and $z =a$ is a projection.

Observe next that  $$\| P(x)  \pm (\mathbf{1}- z)  \| = \| P(x)  \pm (I-P)(\mathbf{1})  \| =
 \max\{\|P(x) \|, 1 |\}= 1$$ for any $x \in {\rm Ball}(A)$.  Hence $P(x) \in z A z$ by Lemma \ref{orttr}.   Similarly $(I-P)(A) \subset (\mathbf{1}-z) A (\mathbf{1}-z)$.  Since $A = P(A) + (I-P)(A)$ we have
 $A =   zAz +  (\mathbf{1}-z) A (\mathbf{1}-z)$, from which it follows that $z$ is central in $A$ in the C$^*$-product.  We also have
  $0 = z (I-P)(x) = zx - P(x)$ for $x \in A$, so that $P$ is multiplication by $z$.  We have in fact now proved (2) for unital $A$. 

If $J$ is an $M$-ideal in an approximately unital real operator algebra  (resp.\ Jordan operator algebra) $A$ then it follows
from what we just proved that  $$J^{**}\cong \overline{J}^{w^*} =J^{\perp \perp} = eA^{**} = e\circ A^{**}$$ for a central projection $e$ in $A^{**}$.  Therefore, $J^{\perp \perp}$ is a complete $M$-summand (see the proof of (2) for more detail). By \cite[Lemma 4.1]{BT} $J$ is approximately unital, and hence Proposition 4.8 in \cite{BT} 
assures us that $J$ is an ideal  (resp.\ Jordan ideal). Also,  $J$ is  a complete $M$-ideal. 
 
 The converse is similar.  If $J$ is a closed Jordan ideal in an approximately unital real Jordan subalgebra $A$ of C$^*$-algebra $B$ then
 by Proposition 4.8 in  \cite{BT} 
 we have  $J^{\perp \perp} = eA^{**}$  for a central projection $e$ in $A^{**}$
as 
above.   In the C$^*$-product we have
e.g.\ $a j \in eA^{**} \cap B = J^{\perp \perp}  \cap B = J$, for $a \in A, j \in J$. 

(2) \ That 
the central projections in  JM$(A)$ coincide with the projections in M$(A)$ commuting with $A$, is evident. 
Indeed a  central projection in  JM$(A)$ commutes with $A$ and hence with $A^{**}$,
and  in the C$^*$-product $ea  = e \circ a = ae \in A$ for $a \in A$.  So $e$ is in M$(A)$.  The converse is similar.

Let $P$ be an $M$-projection on an approximately unital Jordan operator algebra $A$.  Then $P^{**}$ is an $M$-projection on $A^{**}$, so that by the above there is a 
central projection in $A^{**}$ with $P^{**}(\eta) = e \eta$ for $\eta \in A^{**}$.   Then 
$$eA = Ae = P^{**}(A) = P(A) \subset A,$$ and so $e \in {\rm M}(A)$ and 
$P(A) = eA$.  It can be easily seen (in several ways) that $P(A)$ is a real complete  $M$-summand of $A$.   
 For example,  multiplication by $e$ on $A$ is clearly a real adjointable (indeed selfadjoint) left multiplier of $A$
 in the sense of 
\cite{Sharma,BReal}.  Similarly it is a real adjointable right multiplier, so that it is in the centralizer algebra from \cite{BReal}, and
hence  is a complete $M$-projection. We can also simply check that for each natural $n$, the mapping given by entry-wise amplification $P_n : M_n (A)\to M_n(A)$ satisfies $P_n (x) = e_n x$, where $e_n = e\oplus\ldots \oplus e$ is a central projection in $M_n (A^{**})$.  Hence $P_n$ is an $M$-projection. So $P(A)$ is a  complete $M$-summand. 
\end{proof}

From the last theorem we  recover Corollary 4.2 and a particular case of Corollary 4.1 in \cite{BNPS}, by a quite 
different argument. We must admit however that the techniques here are not valid for more general structures like real TROs, real JB$^*$-algebras, and, more generally, real JB$^*$-triples. 

\begin{corollary}\label{c M-ideals of real Cstar-algebras}
Let $M$ be a 
closed subspace of a real C$^*$-algebra (respectively, a real JC$^*$-algebra) $A$. Then $M$ is an $M$-ideal of $A$ if and only if it is an ideal (respectively, Jordan ideal) of $A$.
\end{corollary}

{\bf Remark.}  Of course an immediate corollary of the theorem is the characterization of $L$-summands of the dual of the algebras $A$ considered in that theorem, or of the predual of $A$ if such exists. As we said in the Introduction 
$L$-summands are the ranges of $L$-projections, which in turn are the preduals of $M$-projections.  So by the theorem,  $L$-summands are the ranges of the map $\psi \mapsto e \psi$ for the central projections $e$ considered in the theorem.  

\bigskip

Let $A$ be a real Jordan operator algebra.  As in the complex case one defines a projection $p \in A^{**}$ to be $A$-{\em open} if 
$p \in (p A^{**} p \cap A)^{\perp \perp}$.   (See  \cite{Tepsan} for this and most of the following facts, which follow in the real case
by the same arguments for the  complex case from \cite{BWjmn}.) This is equivalent to $p$ being $A^1$-open.  We call $p A^{**} p \cap A$ a {\em hereditary subalgebra} (or HSA) of $A$, and we call $p$ its {\em support (projection)}, which is an identity for the bidual of the HSA. 
If $D$ is a HSA in $A$, then $D_c$ is a HSA in  $A_c$. In addition, a projection is open or closed if and only if it is open or closed with respect to $D_c$.
A HSA in $A$ is an approximately unital Jordan subalgebra.   Indeed it is an inner ideal of $A$ in the Jordan sense, that is, $dad \in D$ for $d \in D, a \in A$.  Indeed we have as in the complex case  \cite[Proposition 3.3]{BWjmn} and with the same proof:

\begin{lemma} \label{isinn}  A closed subspace $D$ of a real Jordan operator algebra $A$ is a HSA of $A$ if and only if 
$D$ is an approximately unital inner ideal of $A$ in the Jordan sense.   Conversely, if $E$ is a 
subspace of 
$A$ such that $E^{\perp \perp} = p A^{**} p$ for a projection $p \in A^{**}$  then $E$ is a HSA and $p$ is its support projection, an open projection.  \end{lemma}
 
 It follows as a corollary that a closed approximately unital Jordan ideal $J$ in a real Jordan operator algebra $A$ is a HSA.   These are exactly the $M$-ideals  in Theorem \ref{mid} (1).  
 
 \begin{corollary} 
\label{isce} Let $A$ be an approximately unital real (Jordan) operator algebra.   The $M$-ideals in $A$ are in bijective correspondence with 
the open projections for $A$ which are central in $A^{**}$. \end{corollary}

\begin{proof} According to Theorem \ref{mid}, and its proof,  the $M$-ideals $J$ in $A$ are in bijective correspondence with the central projections $e\in A^{**}$ for which 
 the $M$-summand $J^{\perp \perp}$ equals $e A^{**}$.  Such $e$ is simply the support projection of $J$.
 \end{proof}

 We recall that  Alfsen and Effros characterized $M$-ideals in the selfadjoint part of a complex C$^*$-algebra.  We extend this to the real case. 
 
 \begin{corollary}
\label{exce} Let $A$ be a real C$^*$-algebra.   The $M$-ideals in $A_{\rm sa}$ are exactly the $J_{\rm sa}$ for a closed two-sided ideal $J$ in $A$. \end{corollary}

\begin{proof}   No doubt Alfsen and Effros' proof works in the real case after using our theorem above, however it uses facts about the centralizer algebra.  Instead we will deduce the result from the  open projection technology above.
Certainly such $J_{\rm sa}$ is a closed Jordan ideal in the JC-algebra $A_{\rm sa},$ and hence an $M$-ideal by the above. Observe that $A_{\rm sa}$ is a JC-algebra, a real JC$^*$-algebra, and a real JB$^*$-triple. Conversely, if $K$ is an $M$-ideal in $A_{\rm sa}$ then by our theorem above it is a closed Jordan ideal in the JC-algebra $A_{\rm sa}$. Clearly, the open projection $p$ supporting $K$ in Corollary \ref{isce} is also open with respect to $A$, so it is also the support projection of a closed ideal $J$ in $A$. Clearly 
$$J \cap A_{\rm sa} = p A^{**} p \cap A_{\rm sa} =  \{ a \in A_{\rm sa} : a = pap \} = K$$
as desired.       
\end{proof} 

There is an alternative argument to see the last corollary holds even in case that $A$ is a real JB$^*$-algebra. If $K$ is an $M$-ideal of $A_{\rm sa}$ and the latter is regarded as a real JB$^*$-triple, Theorem 5.2 in \cite{BNPS} assures us that $K$ is a triple ideal of $A_{\rm sa}$. A routine argument shows that $K\oplus i K$ is an $M$-ideal of the complexification $(A_{\rm sa})_{c}$ of $A_{\rm sa}$ (cf. \cite[Corollary 5.1]{BNPS}). Since $A_{\rm sa}$ also is a JB-algebra, $(A_{\rm sa})_{c}$ becomes a JB$^*$-algebra and thus $K\oplus i K$ is a self-adjoint Jordan ideal of $(A_{\rm sa})_{c}$ (see, for example, \cite[page 8]{GarPe2021}).  This  implies that $K$ is a Jordan ideal of $A_{\rm sa}$.
In this setting we do not need to appeal to Corollary \ref{isce} above, actually a celebrated result by Edwards (see \cite[Theorem 3.3]{Ed77}) assures the existence of an open projection $p\in A_{\rm sa}^{**}$ satisfying $K= A_{\rm sa}\cap (A_{\rm sa}^{**}\circ p)$. As before, it suffices to take $J = A\cap (A^{**}\circ p)$ to get a Jordan ideal in $A$ with $K = J_{\rm sa}.$ Corollary \ref{isce} can be in fact regarded as an extension of Edwards' result, after of course using our main theorem.

\begin{corollary}
\label{isquo} Let $A$ be an approximately unital real operator algebra  (resp.\ Jordan operator algebra), and $J$ a closed subspace of $A$.   Then 
$J$ is an $M$-ideal in $A$ if and only if $J$ is an $M$-ideal in $A^1$, and if and only if $J + iJ$ is an $M$-ideal in $A_c$. \end{corollary} 

\begin{proof}   We prove just the more difficult case.
 If $J$ is a closed Jordan ideal in approximately unital real Jordan operator algebra $A$ then $J^{\perp \perp} = eA^{**}$  for a central projection $e$ in $A^{**}$.
So $(J + i J)^{\perp \perp} = eA_c^{**}$.   So $J + iJ$ is an $M$-ideal in $A_c$.  Conversely, if $J+iJ$ is a closed Jordan ideal in  $A_c$
then $a j + ja \in A \cap (J+iJ) = J$  for $a \in A, j \in J$.   So $J$ is a closed Jordan ideal in  $A$.   Suppose that $(J + i J)^{\perp \perp} = eA_c^{**}$ with $e =  f + ig$ for $f,g \in J^{\perp \perp} 
\subset A^{**}$.   Then $a e = af + i ag = a$ in the C$^*$-product for $a \in J^{\perp \perp}$.   So $f$ is an identity for $J^{\perp \perp}$, so that $J$ is approximately unital (cf. \cite[Lemma 4.2]{BT}).  \end{proof}

Not all $M$-ideals in $A_c$ need be of the form $J + iJ$ above.  For a counterexample consider the canonical real C$^*$-subalgebra $$C(\bT,\bC)^{\tau}= \left\{f\in C(\bT,\bC): \tau (f) (t) = \overline{f(\overline{t})} \ (t\in \bT) \right\},$$ which is not a real $C(K)$ space.  The ideal $\mathcal{J}$ in $A_c=C(\bT,\bC)$ corresponding to the functions annihilating on the upper half circle is not of the form $J + iJ$. Indeed the ideals of form $J + iJ$ above are exactly the ideals invariant under the conjugate-linear $*$-automorphism $\tau$.

Just as in the complex case in \cite[Section 3.3]{BWjmn} we obtain: 

\begin{corollary}
\label{isMin}  If $J$ is an $M$-ideal in 
an approximately unital real Jordan operator algebra $A$ then $A/J$ is 
(completely isometrically Jordan isomorphic to) an approximately unital real Jordan operator algebra.
\end{corollary}

\begin{corollary} $M$-projections on an approximately unital real operator algebra $A$ are  completely real positive as maps on $A$ in the sense of
{\rm \cite{BT}}, 
and take selfadjoint (resp.\ skew) elements of $A$ to selfadjoint (resp.\ skew) elements. \end{corollary}

\begin{proof} By Theorem~\ref{mid} above, every $M$-projection $P$ on $A$ is multiplication operator by a fixed central projection $e$ in M$(A)$. From this it is easy to see that $P$ is 
completely real positive and takes selfadjoint (resp.\ skew) elements to selfadjoint (resp.\ skew) elements.  
 \end{proof}
 
{\bf Remark.} For an $M$-projection $P$  on a unital real operator algebra $A$ we have the stronger condition that 
$\| \mathbf{1} - x \| \leq 1$ implies that $\| \mathbf{1} - P(x) \| = \max \{ \| P(\mathbf{1}-x ) \| , \| \mathbf{1}-P(\mathbf{1}) \| \} = 1$,
since $P(\mathbf{1})$ is a projection.   
Also, $P$ is a real positive element of $B(A)$. 
Real positive projections are studied in  \cite{BNjp} (see also \cite{BPosx,BT}), with many general results given there.

\section{One-sided $M$-ideals} \label{osmi}

One-sided  $M$-summands and one-sided $M$-ideals were defined in \cite{BEZ} in the complex case, and in \cite{Sharma}
in the real case.   See also e.g.\  \cite[Section 4.8]{BLM} and \cite{BZ}
for the complex case, and \cite{BReal} in the real case.  Recall that a \emph{complete left $M$-projection} on a real operator space $X$, is an idempotent map $P$ on $X$ such that the map $\nu_{_P}^{c} :  X \to  C_2 (X)$, $x \mapsto \left[\begin{matrix} P(x) \\
	x - P(x)
\end{matrix} \right]$ is a complete isometry. A subspace $J$ of  $X$ is a \emph{right $M$-summand} (resp.\ \emph{right $M$-ideal})
 if $J$  (resp.\ $J^{\perp \perp}$)  is the range of a complete left $M$-projection on $X$ (resp.\ $X^{**}$). 
 In an approximately unital real  operator algebra $A$ the  right  
 $M$-ideals 
 are easier to characterize (see \cite[Section 5]{Sharma}) than the two-sided case in our previous section.   Indeed the right  
 $M$-ideals are  the closed right ideals with a left cai.   
(The just quoted reference did not however characterize the right $M$-summands in the nonunital case.)  They also 
 correspond just as in the complex case to the open projections in $A^{**}$ 
  (see e.g.\ the discussion below Corollary 4.9 in \cite{BT}). 
 
One may ask if, similarly to the previous section and \cite{BNPS}, closed one-sided ideals in C$^*$-algebras may be characterized 
by a Banach space condition, or without using the operator space structure.   However this is impossible, since 
for example the transpose map is a surjective isometry on $M_n$ yet switches `left' and `right'. 
  In the complex case Theorem 1.1 in \cite{BSZ}  
  (due to the first author, Roger Smith and Zarikian) comes close to a Banach space characterization. 
  This theorem gives  a list of conditions, most of which involve $1 \times 2$ or $2 \times 1$ matrices, each of which 
  characterize closed one-sided $M$-summands. 
It is tempting to presume that these characterizations hold in the real case.   In particular on first inspection one would expect,  based on the complex case, that an idempotent linear map $P : A \to A$ on a  real C$^*$-algebra $A$ is a left $M$-projection, that is of form $P(x) = ex$ for a projection $e$, if and only if $\tau_P$ is contractive, where 
$$\tau_P \left(  \begin{bmatrix} x \\  y \end{bmatrix}  \right) = \left(  \begin{bmatrix} Px \\  y \end{bmatrix}  \right) .$$
In fact many of the items in  \cite[Theorem 1.1]{BSZ} fail to characterize left $M$-projections on real C$^*$-algebras, as we show in the remark after the next theorem.

Let $X$ be an operator space and $x, y \in X$. We recall from 
\cite{BSZ} that
 $x$ and $y$ are \emph{left orthogonal} (written $x \perp_L y$) if there exists a complete isometry $\sigma:X \to B(H)$, for a Hilbert space $H$,
  such that $\sigma(x)^*\sigma(y) = 0$.  As in \cite[Theorem 5.1]{BEZ}, an   idempotent linear map  $P$ on $X$ is a  left $M$-projection if and only if $P(X) \perp_L  (I-P)(X)$. 
The equivalence of items  (i) and (iv) in \cite[Theorem 1.1]{BSZ}  also holds in the real case, indeed even in not-necessarily selfadjoint associative or Jordan operator algebras: 

\begin{theorem} \label{anything}
Let $\mathcal{B}$ be a real or complex approximately unital operator algebra or Jordan operator algebra,  and let $P:\mathcal{B} \to \mathcal{B}$ be an idempotent linear map. The  following conditions are equivalent:
  \begin{enumerate}[$(1)$] 
\item $P$ is a left $M$-projection,
\item 
$P(x) = e x$ for a projection $e$ of norm $1$ in
${\rm M}(\mathcal{B})$, 
\item $\| [ P(x) \;  \; (I - P)(y) ] \|_{R_2(\mathcal{B})} = \max \{  \| P (x) \|,  \| (I-P) (y) \| \}$ for all $x, y \in \mathcal{B},$\item  $P(\mathcal{B})^*   (I-P)(\mathcal{B}) = 0$ inside some (equivalently,  every) C$^*$-cover of $\mathcal{B}$. 
\end{enumerate} 
Also,  $P$ is an $M$-projection if and only if $$\left\|  \begin{bmatrix} P (x) \\  (I-P) (y) \end{bmatrix} \right\|_{C_2(\mathcal{B})} = \| [ P(x) \;  \; (I - P)(y) ] \|_{R_2(\mathcal{B})} = \max \{  \| P (x) \|,  \| (I-P) (y) \| \},$$ for all $x, y $ in $\mathcal{B}$.
\end{theorem}

\begin{proof} 
(1) $\Rightarrow$ (2)\  Since the projections in $\Ml(X)$  are exactly the left $M$-projections on $X$ of \cite{BEZ,Sharma}, 
this follows from  Theorem \ref{tlma} and Lemma \ref{isbab} (and their original complex variants).  

(2) $\Rightarrow$ (3)\ 
This is  just as in the complex operator algebra case.

(3) $\Rightarrow$ (2)\  This is essentially also as in the complex case.  That is, the proof of \cite[Corollary 6.5]{BSZ} and the previous several supporting
 results carry through essentially 
verbatim, even in the Jordan case,  to give the required projection $e$  in LM$(\mathcal{B})$, hence in M$(\mathcal{B})$ by Lemma \ref{isbab}.   One key point in the proof that needs to 
be amplified is the point (used twice in the argument from \cite{BSZ}) that we can always represent a (real in our case)
 operator space $X \subset B(H)$ in such a way that for every $x \in X$ there exists $\xi \in {\rm Ball}(H)$ with $\| x \| = \| x \xi \|$.  A quick way to 
 see this is to apply the complex case of this to $X_c$, and then use the fact that every complex Hilbert space is a real Hilbert space. 
 
 (2) $\Rightarrow$ (4)\ This is obvious, for every C$^*$-cover of 
 $(\mathcal{B}$.
 
 (4) $\Rightarrow$ (1)\ If  $P(\mathcal{B})^*   (I-P)(\mathcal{B}) = 0$ then $P(\mathcal{B})  \perp_L  (I-P)(\mathcal{B}) = 0$ and so $P$ is a  left $M$-projection
 by the discussion above Theorem \ref{anything}.  
  
 If the last centered condition in the theorem 
 statement holds then since (3) implies (2), and by symmetry we have that $P(x) = e x = xf$ for $x \in \mathcal{B}$, for certain projections $e, f$.
 Taking $x$ in a cai for $\mathcal{B}$ and taking a weak$^*$ limit shows that $e=f$, and it is now easy to see that this is central in M$(\mathcal{B})$. 
 So $P$ is an $M$-projection. 
\end{proof}

\begin{remark}\label{r characterizations in BSZ fail} We exhibit some counterexamples showing  that many of the statements in  \cite[Theorem 1.1]{BSZ} fail to characterize left $M$-projections on real C$^*$-algebras. Indeed let  $\Hdb$ be 
the quaternions, regarded as a real C$^*$-algebra (cf. \cite{Li}), and let $P$ be any nontrivial contractive projection on $\ell^4_2(\Rdb) \cong \Hdb$.   The key point is that the norm of the column with entries $\alpha, \beta \in  \Hdb$ has norm in $C_2( \Hdb)$ equal to  $$\left\|  \begin{bmatrix} \alpha \\  \beta \end{bmatrix} \right\|_{C_2(\Hdb)} = \sqrt{\| \alpha^* \alpha +  \beta^*  \beta \|_{\Hdb}}
= \sqrt{\|  (\| \alpha \|_2^2 + \| \beta  \|_2^2) \mathbf{1} \|_{\Hdb}} = \sqrt{\| \alpha \|_2^2 + \| \beta  \|_2^2 }.$$
It follows that $\tau_P$ is contractive, and $\nu_P^c$ is an isometry, indeed
 $$\left\| \begin{bmatrix} P (\alpha) \\  \beta \end{bmatrix} \right\|_{C_2(\Hdb)}^2 = \| P (\alpha) \|_2^2 + \| \beta  \|_2^2 \leq \left\|  \begin{bmatrix} \alpha \\  \beta \end{bmatrix} \right\|_{C_2(\Hdb)}^2 \; \; , \;  \; \;  \| P (\alpha) \|_2^2 + \| (I-P) (\alpha) \|_2^2 = \| \alpha \|_2^2.$$  
 Similarly (v)  in \cite[Theorem 1.1]{BSZ} says that $\|  P (\alpha) \|^2 \mathbf{1}  \leq \|   \alpha \|^2 \mathbf{1}$, which is valid. 
So (ii),(iii)  and (v) in \cite[Theorem 1.1]{BSZ}  hold, but (i) fails.   However, $\Hdb$ contains no nontrivial projections, 
indeed no selfadjoint elements besides real multiples of $\mathbf{1}$. 

This also shows that e.g.\  the preparatory lemmata in \cite[Section 4]{BSZ} seem to  fail, or have no relevant real version.  Indeed if these results were true in our setting, by combining Lemmas 4.3 and 4.4 there would give the statement that if $\nu_P^c$ is an isometry then $P(x) = hx + xk$ for $h,k \in R_{\rm sa}$.
However, taking $P(\alpha) = \frac{\alpha + \alpha^*}{2}$ on $\Hdb$ then $\nu_P^c$ is an isometry.   If $P(x) = hx + xk$ as above then we have $P(\mathbf{1}) = h+k = \mathbf{1}$, and then it is easily seen that $0 = P(i)$ gives a contradiction.
\end{remark}

\begin{remark} Corollary  6.7 in \cite{BSZ} gives a characterization of  left $M$-projections  in TROs which is similar to the above.    Namely, 
let $X$ be a TRO (or, equivalently, a Hilbert C$^*$-module) and let 
$P:X \to X$ be an idempotent linear map. Then $P$ is a 
left $M$-projection 
if and only if  $$\| [ P_n (x) \;  \; (I - P_n)(y) ] \|_{R_2(X)} = \max \{  \| P_n (x) \|,  \| (I-P_n) (y) \| \}, \ x, y \in M_n(X), n \in \mathbb{N}.$$   If $X$ is a right Hilbert C$^*$-module say, then such $P$ are exactly 
the adjointable projections on $X$. 
 We leave the proof in the real case to the reader.   It is verified on the last page of \cite{BReal} that the characterizations 
of left $M$-summands and left $M$-ideals in a TRO in  \cite[Lemma 8.5.16 (iv) and Theorem 8.5.19]{BLM} (see also \cite[Theorem 6.6]{BEZ}),
  are true in the real case.
\end{remark}

\begin{corollary}
\label{isMin2} Let $A$ be an approximately unital real operator algebra and $J$ a closed subspace of $A$.   Then 
$J$ is a right  $M$-ideal in $A$ if and only if $J$ is a right  $M$-ideal in $A^1$, and if and only if $J + iJ$ is a right $M$-ideal in $A_c$. 
\end{corollary} 

\begin{proof}   This is similar to Corollary \ref{isquo}  so we prove just one case (see also \cite[Corollary 5.9]{Sharma} for a more general but more complicated result).
 If  $J+iJ$ is a closed right ideal in  $A_c$
then $ja  \in A \cap (J+iJ) = J$ for $a \in A, j \in J$.   So $J$ is a closed right ideal in  $A$.   Suppose that $(J + i J)^{\perp \perp} = eA_c^{**}$ with $e =  f + ig$ for $f,g \in J^{\perp \perp} 
\subset A^{**}$.   Then $ea  = fa + i ga = a$  for $a \in J^{\perp \perp}$.   So $f$ is  a left  identity for $J^{\perp \perp}$, so that $J$ 
has a left cai by a standard argument seen above.  \end{proof} 

One may now proceed to check which other results in the one-sided $M$-ideal  theory (see e.g.\ \cite{BEZ,BZ,BLM}) work in the real case.
Some  part of this tedious  process has already been done in \cite{Sharma,BReal}, and it quickly becomes apparent that nearly all of the theory should work.
We content ourselves here by mentioning 
the real case of  three  results of independent interest from \cite[Section 5.2]{BZ}, whose proofs require a small modification.
We leave  other results in that source to the interested reader. 
 
 \begin{definition}\label{IV.B.1}
Let $X$ be a real or complex  operator space, $Y$ a closed linear subspace of $X$, and $T:X \to X$ a linear map such that
$T(Y) \subset Y$. By $T|_Y:Y \to Y$ we denote the restriction of $T$ to $Y$, and by $T/Y:X/Y \to X/Y$ we denote the (well-defined) quotient linear map
 $$(T/Y)(x + Y) = Tx + Y.$$
\end{definition}

As in the complex case \cite{BEZ},  if $T$ is completely bounded, then so are $T|_Y$ and $T/Y$, and
$\|T|_Y\|_{cb}, \|T/Y\|_{cb} \leq \|T\|_{cb}$.   In the following, $\Ml(X)$ (resp.\ $\Al(X)$) is the real operator algebra
(resp.\ real C$^*$-algebra) 
 of {\em real operator space multipliers} (resp.\  \emph{left adjointable multipliers}), 
 introduced above Theorem \ref{tlma}.  See e.g.\ \cite[Section 4]{BReal} and references therein.

\begin{proposition} \label{IV.B.2}
Let $X$ be a real operator space, $Y$ a closed linear subspace of $X$, and $T:X \to X$ a real linear map such that
$T(Y) \subset Y$.
\begin{enumerate}[{\rm(i)}]
\item If $T \in \Ml(X)$, with $\|T\|_{\Ml(X)} \leq 1$, then $T|_Y \in \Ml(Y)$ and $T/Y \in \Ml(X/Y)$,
with $\|T|_Y\|_{\Ml(Y)} \leq 1$ and $\|T/Y\|_{\Ml(X/Y)} \leq 1$.
\item If $T \in \Al(X)$ and $T^\star(Y) \subset Y$, then $T|_Y \in \Al(Y)$ and $T/Y \in \Al(X/Y)$.
Furthermore, $(T|_Y)^\star = T^\star|_Y$ and $(T/Y)^\star = T^\star/Y$.
\end{enumerate}
\end{proposition}

\begin{proof}   One may appeal to 
the proof of \cite[Proposition 5.8]{BEZ} up to the point where 
it is shown that the set of $T  \in \Al(X)$  such that $T(Y)   \subset Y$ and $T^\star(Y)   \subset Y$ 
is a $*$-subalgebra $B$, and that $T \mapsto T/Y$ is a norm decreasing unital homomorphism $\theta$ from $B$ into $\Ml(X/Y )$. 
Unfortunately the next step uses a selfadjoint element trick that fails in the real case. 
However instead we appeal to  \cite[Theorem 2.6]{BT}, which yields that $\theta$  is a $*$-homomorphism into 
$\Delta(\Ml(X/Y )) = \Al(X/Y )$.  The rest is clear.  \end{proof}

Since the projections in $\Ml(X)$ and $\Al(X)$ are exactly the left $M$-projections on $X$ (cf.\ \cite[Proposition 4.1]{BEZ}, \cite[Proposition 5.3]{Sharma}), as 
corollaries one obtains hereditary and quotient properties of right $M$-summands (resp.\ right $M$-ideals).  This is just as in \cite[Propositions 5.3 and 5.9]{BZ} but using Proposition \ref{IV.B.2} above:

\begin{corollary} \label{IV.B.3}
Let $X$ be a real operator space, $J$ a closed linear subspace of $X$, and let $Y$ be a closed linear subspace of $J$.
\begin{enumerate}[{\rm(i)}]
\item If $Y$ is a right $M$-summand of $X$, then $Y$ is a right $M$-summand of $J$.
\item If $Y$ is a right $M$-ideal of $X$, then $Y$ is a right $M$-ideal of $J$.
\end{enumerate}
\end{corollary}

{\bf Remark.} The converse of the last result fails in general unlike 
for classical $M$-ideals: If $J$ is a right summand of $X$  and  $Y$ is a right summand of $J$, then $Y$ need not be a right summand in $X$.
See \cite[Example 5.2.4]{BZ}.

\begin{corollary} \label{IV.B.5}
Let $X$ be a real operator space, $J$ be a closed linear subspace of $X$, and $Y$ be a closed linear subspace of $J$.
\begin{enumerate}[{\rm(i)}]
\item If $J$ is a right $M$-summand of $X$, then $J/Y$ is a right $M$-summand of $X/Y$.
\item If $J$ is a right $M$-ideal of $X$, then $J/Y$ is a right $M$-ideal of $X/Y$.
\end{enumerate}
 \end{corollary}

\noindent\textbf{Acknowledgements}\smallskip

\noindent  We dedicate this paper to Roger Smith for the  occasion of his belated 70th birthday.  We  thank him for his warm friendship 
 over the years, and also for some history of  the early days of $M$-ideals. \\

D.P. Blecher acknowledges support by a Simons Foundation Collaboration Grant and NSF Grant DMS-2154903. M. Neal was supported
by Denison University.  A.M. Peralta supported by grant PID2021-122126NB-C31 funded by MICIU/AEI/ 10.13039/501100011033 and by ERDF/EU, by Junta de Andalucía grant FQM375, IMAG--Mar{\'i}a de Maeztu grant CEX2020-001105-M/AEI/10.13039/501100011033 and (MOST) Ministry of Science and Technology of China grant G2023125007L.
S. Su supported by grant PID2021-122126NB-C31 funded by MICIU/AEI/10.13039/ 501100011033 and by China Scholarship Council Program (Grant No.202306740016).


\subsection*{Data availability}

There is no data set associated with this submission.

\subsection*{Statements and Declarations}

The authors declare they have no financial nor conflicts of interests.

\end{document}